\documentclass[twoside, a4paper, 12pt]{amsart}
\usepackage{amsaddr}
\usepackage{fullpage}
\usepackage{amsfonts}
\usepackage{amssymb}
\usepackage{enumitem}
\usepackage{scalerel}
\theoremstyle{plain}
\newtheorem*{claim*}{Claim}
\newtheorem{thm}{Theorem}[section]
\newtheorem{cor}[thm]{Corollary}
\newtheorem{lem}[thm]{Lemma}
\newtheorem{prop}[thm]{Proposition}
\theoremstyle{definition}
\newtheorem{defn}[thm]{Definition}
\newtheorem{ex}[thm]{Example}

\newcommand{\ar}{\mbox{$\mathcal{R}$}}
\newcommand{\arn}{\mbox{$\mathcal{R}$}\text{-noetherian}}
\newcommand{\el}{\mbox{$\mathcal{L}$}}
\newcommand{\h}{\mbox{$\mathcal{H}$}}
\newcommand{\dee}{\mbox{$\mathcal{D}$}}
\newcommand{\jay}{\mbox{$\mathcal{J}$}}
\newcommand{\leqr}{\leq_{\mathcal{R}}}
\newcommand{\Z}{\mathbb{Z}}
\newcommand{\N}{\mathbb{N}}
\newcommand{\D}{\Diamond}
\newcommand{\E}{\text{End}}
\newcommand{\SE}{\text{SEnd}}
\newcommand{\M}{\mathcal{M}}
\renewcommand{\P}{\mathcal{P}}
\renewcommand{\S}{\mathcal{S}}
\renewcommand{\a}{\alpha}
\renewcommand{\b}{\beta}
\newcommand{\g}{\gamma}
\newcommand{\f}{\varphi}
\newcommand{\s}{\rtimes_{\varphi}}

\begin{document}
\subjclass[2020]{20M12, 20M10}
\title{\large{The ascending chain condition on principal right ideals for semigroup constructions}}
\author{Craig Miller}
\address{Department of Mathematics, University of York, YO10 5DD, UK\\
craig.miller@york.ac.uk}
\maketitle
\vspace{-1em}
\begin{abstract}
We call a semigroup {\em $\ar$-noetherian} if it satisfies the ascending chain condition on principal right ideals, or, equivalently, the ascending chain condition on $\ar$-classes.  We investigate the behaviour of the property of being $\arn$ under the following standard semigroup-theoretic constructions: semidirect products, Sch{\"u}tzenberger products, free products, Rees matrix semigroups, Brandt extensions, Bruck-Reilly extensions and semilattices of semigroups.
\end{abstract}
~\\
\textit{Keywords}: Semigroup, principal right ideal, ascending chain condition.\\
\textit{Mathematics Subject Classification 2010}: 20M10, 20M12.
\vspace{1em}

\section{\large{Introduction}\nopunct}
\label{sec:intro}

A {\em finiteness condition} for a class of universal algebras is a property that is satisfied by at least all finite members of that class.  The study of finiteness properties was pioneered by Noether in the early 20th century in the context of ascending chain conditions on rings \cite{Noether}, and has become an established theme in many algebraic disciplines.  The main motivation is to develop a better insight into the structure of the objects of study, and, in particular, to get a sense of how different they are to finite objects.

This article is concerned with the class of semigroups and the finiteness condition of satisfying the ascending chain condition on principal right ideals.  We call semigroups satisfying this condition {\em $\ar$-noetherian}\footnote{$\ar$-noetherian semigroups are also known in the literature as `ACCPR-semigroups' or `semigroups satisfying ACCPR'.}, owing to the fact that it is equivalent to the ascending chain condition on $\ar$-classes ($\ar$ is one of the five Green's relations on a semigroup).

The property of being $\ar$-noetherian has a natural analogue in ring theory: the ascending chain condition on principal right ideals of rings.  Indeed, the study of $\ar$-noetherian semigroups was initiated in a paper \cite{Liu} investigating the ascending chain condition on principal ideals of rings of generalised power series.  The article \cite{Mazurek} built on this work by characterising the ascending chain on principal right (and left) ideals for the more general class of skew generalised power series rings $R[[S,\omega]]$ (with coefficients in a ring $R$ and exponents in a strictly totally ordered monoid $S$), and here again the property of $S$ being $\ar$-noetherian is crucial \cite[Theorem 3.3]{Mazurek}.  This work motivated the paper \cite{Stopar}, which considers the ascending chain conditions on principal right and left ideals of semidirect product of semigroups and makes a connection with the corresponding properties for rings of skew generalised power series.

A stronger condition than that of being $\arn$ is the property of satisfying the ascending chain condition on {\em all} right ideals; we call semigroups satisfying this condition {\em weakly right noetherian}\footnote{In the literature, `right noetherian' is the standard name given to semigroups that satisfy the ascending chain condition on right congruences.  However, weakly right noetherian semigroups have occasionally been termed `right noetherian', while right noetherian semigroups have been called `strongly right noetherian'.}.  Such semigroups have received significant attention; see for instance \cite{Aubert,Davvaz,Jespers,Miller:2021}.  The property of being weakly right noetherian can be characterised in terms of principal right ideals: a semigroup $S$ is weakly right noetherian if and only if it is $\arn$ and contains no infinite antichain of principal right ideals (or, equivalently, $S$ contains no infinite strictly ascending chain or infinite antichain of $\ar$-classes) \cite[Theorem 3.2]{Miller:2021}.  Both the properties of being $\arn$ and being weakly right noetherian were considered in the author's recent article \cite{Miller:2023}, of which the main purpose was to study the relationship between a semigroup and its one-sided ideals with respect to each of these properties.

The purpose of the present paper is to investigate the behaviour of the property of being $\arn$ under various semigroup-theoretic contructions.  (Of course, our results will also have left-right duals for the property of being $\el$-noetherian.)  After introducing the necessary preliminary material in Section \ref{sec:prelim}, we consider semidirect products in Section \ref{sec:sdp}, Sch{\"u}tzenberger products in Section \ref{sec:sp}, free products in Section \ref{sec:fp}, Rees matrix semigroups and Brandt extensions in Section \ref{sec:Rees}, Bruck-Reilly extensions in Section \ref{sec:BR}, and semilattices of semigroups in Section \ref{sec:sos}.

\section{\large{Preliminaries}\nopunct}
\label{sec:prelim}

Throughout this section, $S$ will denote a semigroup. 
We denote by $S^1$ the monoid obtained from $S$ by adjoining an identity if necessary (if $S$ is already a monoid, then $S^1=S$), and we denote by $S^0$ the semigroup with zero obtained by adjoining a zero if necessary.

A subset $I\subseteq S$ is said to be a {\em right ideal} of $S$ if $IS\subseteq I.$  Left ideals are defined dually, and an {\em ideal} of $S$ is a subset that is both a right ideal and a left ideal.

A right (resp.\ left) ideal $I$ of $S$ is said to be {\em generated by} $X\subseteq I$ if $I=XS^1$ (resp.\ $I=S^1X$).  A right (resp.\ left) ideal is said to be {\em finitely generated} if it can be generated by a finite set, and {\em principal} if it can be generated by a single element.

Principal (one-sided) ideals determine the five Green's relations on a semigroup: $\ar,$ $\el,$ $\h,$ $\dee$ and $\jay.$  In this paper we are only concerned with the relation $\ar.$
Green's preorder $\leqr$ on $S$ is given by
$$a\leqr b\,\Leftrightarrow\,aS^1\subseteq bS^1,$$
and this leads to the relation $\ar$:
$$a\,\ar\,b\;\Leftrightarrow\;[a\leqr b\text{ and }b\leqr a]\;\Leftrightarrow\;aS^1=bS^1.$$

When we need to distinguish between Green's relation $\ar$ on different semigroups, we will write the semigroup as a subscript, i.e.\ $\ar_S$ for $\ar$ on $S.$  For convenience, we will write $\leq_S$ rather than $\leq_{\mathcal{R}_S}$, and $a<_S b$ if $a\leq_Sb$ but $(a, b)\notin\ar_S.$

Green's pre-order $\leqr$ induces a partial order on the set of $\ar$-classes of $S.$  We note that the poset of $\ar$-classes of $S$ is isomorphic to the poset of principal right ideals of $S$ (under containment).

A poset $P$ is said to satisfy the {\em ascending chain condition} if every ascending chain $$a_1\leq a_2\leq\cdots$$ of elements of $P$ eventually stabilises. 
We say that $S$ is {\em $\ar$-noetherian} if its poset of principal right ideals satisfies the ascending chain condition.
The following result provides a number of equivalent formulations for a semigroup to be $\arn$.

\begin{prop}\cite[Proposition 2.3]{Miller:2023}
\label{prop:equiv_conditions}
The following are equivalent for a semigroup $S$:
\begin{enumerate}
\item $S$ is $\arn$;
\item every non-empty set of principal right ideals of $S$ contains a maximal element;
\item the poset of $\ar$-classes of $S$ satisfies the ascending chain condition;
\item every non-empty set of $\ar$-classes of $S$ contains a maximal element;
\item $S$ contains no infinite strictly ascending chain of elements under the $\ar$-preorder.
\end{enumerate}
\end{prop}

\begin{cor}\cite[Corollary 2.10]{Miller:2023}
\label{cor:min_right_ideals}
Any semigroup $S$ (with zero) that is a union of (0-)minimal right ideals is $\arn$.  In particular, all completely (0-)simple semigroups and all null semigroups are $\arn$.
\end{cor}

For any non-empty set $X,$ recall that the {\em free semigroup} on $X,$ denoted by $X^+$, is the set of all words over $X,$ and the {\em free monoid} on $X,$ denoted by $X^*$, is $X^+$ with an identity adjoined.  Clearly, for any $u, v\in X^+$, we have $uX^*\subseteq vX^*$ if and only if $v$ is a subword of $u.$  We deduce that:

\begin{prop}
For any non-empty set $X,$ both $X^+$ and $X^*$ are $\arn.$
\end{prop}

Since every semigroup is the quotient of a free semigroup, and there certainly exist semigroups that are not $\arn,$ the property of being $\arn$ is not closed under quotients.  It turns out, however, that this property is closed under Rees quotients:

\begin{lem}\cite[Corollary 3.4]{Miller:2023}
Let $S$ be a semigroup and let $I$ be an ideal of $S.$  If $S$ is $\arn$ then so is $S/I.$
\end{lem}

The property of being $\arn$ is also inherited by one-sided ideals:

\begin{prop}\cite[Corollary 3.2]{Miller:2023}
\label{prop:ideal}
Let $S$ be a semigroup and let $I$ be a right/left/two-sided ideal of $S.$  If $S$ is $\arn$ then so is $I.$
\end{prop}

Let $T$ be a subsemigroup of $S.$  
We say that $T$ is {\em $\ar$-preserving} (in $S$) if the $\ar_T$-preorder is the restriction of the $\ar_S$-preorder to $T$; that is,
$$\leq_T~=~\leq_S\cap~(T\times T).$$
The property of being $\arn$ is inherited by $\ar$-preserving subsemigroups:

\begin{prop}
\label{prop:R-preserving}
Let $S$ be a semigroup and let $T$ be an $\ar$-preserving subsemigroup of $S.$  If $S$ is $\arn$ then so is $T.$
\end{prop}

\begin{proof}
Consider an ascending chain $$a_1\leq_Ta_2\leq_T\cdots$$ in $T.$  Then clearly we have an ascending chain $$a_1\leq_Sa_2\leq_S\cdots$$ in $S.$  Since $S$ is $\arn$, there exists $N\in\N$ such that $a_n\,\ar_S\,a_N$ for all $n\geq N.$  Then, since $T$ is $\ar$-preserving in $S,$ we have $a_n\,\ar_T\,a_N$ for all $n\geq N.$  Hence $T$ is $\arn$.
\end{proof}

A semigroup $T$ is called {\em regular} if for every $a\in T$ there exists $b\in T$ such that $aba=a.$  It is well known that regular subsemigroups are $\ar$-preserving; see \cite[paragraph before Corollary 4.7]{Miller:2021} for a proof.  Thus we have:

\begin{cor}
\label{cor:reg}
Let $S$ be a semigroup with a regular subsemigroup $T.$  If $S$ is $\arn$ then so is $T.$
\end{cor}

A subsemigroup $T$ of $S$ is called {\em right unitary} (in $S$) if it satisfies the following condition: for all $a\in T$ and $b\in S,$ if $ab\in T$ then $b\in T.$

Clearly a right unitary subsemigroup is $\ar$-preserving, so by Proposition \ref{prop:R-preserving} we have:

\begin{cor}
\label{cor:ru}
Let $S$ be a semigroup and let $T$ be a right unitary subsemigroup of $S.$
If $S$ is $\arn$ then so is $T.$
\end{cor}

If the complement of a subsemigroup is a left ideal, then the subsemigroup is right unitary, so we have:

\begin{cor}
\label{cor:comp_left_ideal}
Let $S$ be a semigroup with a subsemigroup $T$ such that $S\!\setminus\!T$ is a left ideal of $S.$  If $S$ is $\arn$ then so is $T.$
\end{cor}

We now a introduce a key notion for this paper. 

\begin{defn}
Let $S$ be a semigroup and let $a\in S.$  We say that $b\in S$ is a {\em local right identity} of $a$ if $a=ab.$
\end{defn}

Clearly in a monoid or regular semigroup, every element has a local right identity.  On the other hand, any left cancellative, idempotent-free semigroup (e.g.\ a free semigroup) has {\em no} element with a local right identity.  Note that a semigroup in which no element has a local right identity is $\ar$-trivial, i.e.\ $\ar$ is the identity relation.

\begin{prop}
\label{prop:nolri}
Let $S$ and $T$ be semigroups with a map $\theta : S\to T$ such that $(aS)\theta\subseteq(a\theta)T$ for each $a\in S.$  If $T$ is $\arn$ and has no element with a local right identity, then $S$ is $\arn$ and has no element with a local right identity.
\end{prop}

\begin{proof}
It is clear if $S$ had an element with a local right identity then so would $T,$ so $S$ has no element with a local right identity.  By Proposition \ref{prop:equiv_conditions}, to prove that $S$ is $\arn$ it suffices to show that it contains no infinite strictly ascending chain of elements under the $\ar$-preorder.  So, consider an ascending chain $$a_1\leq_Sa_2\leq_S\cdots$$ in $S.$  Then, for each $i\in\N$, we have $a_i\in a_{i+1}S^1$.  Therefore, by assumption, we have $a_i\theta\in(a_{i+1}S^1)\theta\subseteq(a_{i+1}\theta)T^1$.  Thus, we have an ascending chain
$$a_1\theta\leq_Ta_2\theta\leq_T\cdots$$ in $T.$
Since $T$ is $\arn$ and $\ar$-trivial, there exists $N\in\N$ such that $a_n\theta=a_N\theta$ for all $n\geq N.$  For each $n\geq N,$ we cannot have $a_N\in a_nS,$ for then we would have
$$a_N\theta\in(a_nS)\theta\subseteq(a_n\theta)T=(a_N\theta)T,$$
contradicting the fact that $T$ has no element with a local right identity.  Thus $a_n=a_N$ for all $n\geq N.$  Hence $S$ is $\arn$.
\end{proof}

\begin{cor}
\label{cor:nolri}
Let $S$ and $T$ be semigroups with a homomorphism $\theta : S\to T.$  If $T$ is $\arn$ and has no element with a local right identity, then $S$ is $\arn$ and has no element with a local right identity.
\end{cor}

\begin{prop}\cite[Proposition 3.9]{Miller:2023}
\label{prop:ideal,lri}
Let $S$ be a semigroup, let $I$ be an ideal of $S,$ and suppose that every element of $I$ has a local right identity in $I.$  Then $S$ is $\arn$ if and only if both $I$ and $S/I$ are $\arn.$
\end{prop}

\section{\large{Semidirect Products}\nopunct}
\label{sec:sdp}

Let $S$ and $T$ be semigroups, and let $\f : T\to\E(S)$ be a homomorphism, where $\E(S)$ denotes the monoid of all endomorphisms of $S.$  The image of an element $t\in T$ under $\f$ will be denoted by $\f_t$.  We write $\f_t$ on the left of its argument; i.e.\ $\f_t(s)$ for $s\in S.$  The {\em semidirect product of $S$ and $T$ with respect to} $\f$, denoted by $S\s T,$ is the semigroup with underlying set $S\times T$ and multiplication given by
$$(s,t)(s',t')=(s\f_t(s'),tt').$$
Note that the direct product $S\times T$ is the semidirect product $S\s T$ where $\f_t=id_S$ for all $t\in T.$

The property of being $\ar$-noetherian was considered for semidirect products in \cite{Stopar}.  Several partial characterisations were obtained for a semidirect product $S\s T$ to be $\arn$.  For instance, if both $S$ and $T$ contain at least one idempotent and $\f_t$ is surjective for every $t\in T,$ then $S\s T$ is $\arn$ if and only if $S$ and $T$ are $\arn$ \cite[Theorem 3.13]{Stopar}.

The purpose of this section is to provide necessary and sufficient conditions for a semidirect product to be $\ar$-noetherian.  To this end, we first prove a few lemmas.

\begin{lem}\label{lem:sdp1}
If $S\s T$ is $\arn$, then at least one of $S$ and $T$ is $\arn$.
\end{lem}

\begin{proof}
Suppose for a contradiction that neither $S$ nor $T$ are $\arn$.  Then there exist infinite strictly ascending chains
$$a_1<_Sa_2<_S\cdots\quad\text{and}\quad b_1<_Tb_2<_T\cdots$$
in $S$ and $T,$ respectively.  Then, for each $i\in\N,$ there exist $s_i\in S$ and $t_i\in T$ such that $a_i=a_{i+1}s_i$ and $b_i=b_{i+1}t_i$.  Observe that $b_1=b_{i+1}t_i\dots t_1$ for each $i\in\N.$  We have 
\begin{align*}
\f_{b_1}(a_i)&=\f_{b_1}(a_{i+1}s_i)=\f_{b_1}(a_{i+1})\f_{b_1}(s_i)=\f_{b_1}(a_{i+1})\f_{b_{i+1}t_i\dots t_1}(s_i)\\&=\f_{b_1}(a_{i+1})\f_{b_{i+1}}(\f_{t_i\dots t_1}(s_i)).
\end{align*}
Therefore, we have
$$(\f_{b_1}(a_i),b_i)=(\f_{b_1}(a_{i+1}),b_{i+1})(\f_{t_i\dots t_1}(s_i),t_i).$$
Thus, letting $U=S\s T,$ we have an ascending chain
$$(\f_{b_1}(a_1),b_1)\leq_U(\f_{b_1}(a_2),b_2)\leq_U\cdots$$
in $U.$  Since $U$ is $\arn$, there exists $N\in\N$ such that $(\f_{b_1}(a_n),b_n)\,\ar_U\,(\f_{b_1}(a_N),b_N)$ for all $n\geq N.$  In particular, there exists $(s,t)\in U$ such that $(\f_{b_1}(a_{N+1}),b_{N+1})=(\f_{b_1}(a_N),b_N)(s,t).$  But then $b_{N+1}=b_Nt,$ contradicting that $b_N<_Tb_{N+1}$. 
\end{proof}

\begin{lem}\label{lem:sdp2}
Suppose that $S$ has an element with a local right identity.  If $S\s T$ is $\arn$, then $T$ is $\arn$.
\end{lem}

\begin{proof}
Suppose for a contradiction that $T$ is not $\arn$.  Then there exists an infinite strictly ascending chain $$b_1<_Tb_2<_T\cdots$$ in $T.$  Then, for each $i\in\N,$ there exists $t_i\in T$ such that $b_i=b_{i+1}t_i$.  Let $a\in S$ have a local right identity $s\in S,$ so that $a=as.$  We have 
$$\f_{b_1}(a)=\f_{b_1}(as)=\f_{b_1}(a)\f_{b_1}(s)=\f_{b_1}(a)\f_{b_{i+1}}(\f_{t_i\dots t_1}(s)),$$
and hence
$$(\f_{b_1}(a),b_i)=(\f_{b_1}(a),b_{i+1})(\f_{t_i\dots t_1}(s),t_i).$$
The final part of the proof is essentially the same as that of Lemma \ref{lem:sdp1}.
\end{proof}

\begin{lem}\label{lem:sdp3}
If either $S$ or $T$ is $\arn$ and has no element with a local right identity, then $S\s T$ is $\arn$.
\end{lem}

\begin{proof}
We claim that the projection map $\pi_S : S\s T\to S$ satisfies the condition of Proposition \ref{prop:nolri}.  Indeed, for any $(a,b), (s,t)\in S\s T,$ we have
$$\big((a,b)(s,t)\big)\pi_S=(a\f_b(s),bt)\pi_S=a\f_b(s)=\big((a,b)\pi_S\big)\f_b(s)\in\big((a,b)\pi_S\big)S,$$
as required.  It is clear that the projection map $\pi_T : S\s T\to T$ is a homomorphism.  Hence, the result follows from Proposition \ref{prop:nolri} and Corollary \ref{cor:nolri}.
\end{proof}

\begin{lem}\label{lem:f-chain}
If $a,a'\in S$ and $b,b'\in T$ with $b\in b'T,$ then $a\big(\f_b(S)\big)^1\subseteq a'\big(\f_{b'}(S)\big)^1$ if and only if $a\in a'\big(\f_{b'}(S)\big)^1$.  Moreover, if $b\,\ar_T\,b'$ then $\f_b(S)=\f_{b'}(S).$
\end{lem}

\begin{proof}
If $a\big(\f_b(S)\big)^1\subseteq a'\big(\f_{b'}(S)\big)^1$, then clearly $a\in a'\big(\f_{b'}(S)\big)^1$.  Conversely, suppose that $a\in a'\big(\f_{b'}(S)\big)^1$.  There exists $t\in T$ such that $b=b't.$  Thus, we have 
$$\f_{b}(S)=\f_{b't}(S)\subseteq\f_{b'}(\f_t(S))\subseteq\f_{b'}(S),$$
and hence
$a(\f_b(S))^1\subseteq a'(\f_{b'}(S))^1$.

Now, if $b\,\ar_T\,b'$, a similar argument as above proves that $\f_{b'}(S)\subseteq\f_b(S),$ and hence $\f_b(S)=\f_{b'}(S).$
\end{proof}

Before stating the main result of this section, we first introduce the following definition.

\begin{defn}
Let $S$ and $T$ be semigroups, and let $\f : T\to\E(S)$ be a homomorphism.  A {\em $\f$-chain} in $S$ is an ascending chain of the form
$$a_1\big(\f_{b_1}(S)\big)^1\subseteq a_2\big(\f_{b_2}(S)\big)^1\subseteq a_3\big(\f_{b_3}(S)\big)^1\subseteq\cdots$$
where $a_i\in S,$ $b_i\in T$ and $b_i\in b_{i+1}T$ for all $i\geq1.$
\end{defn}

\begin{thm}\label{thm:sdp}
Let $S$ and $T$ be semigroups, and let $\f : T\to\emph{End}(S)$ be a homomorphism. 
Then $S\s T$ is $\arn$ if and only if either:
\begin{enumerate}
\item $S$ is $\arn$ and has no element with a local right identity; or
\item every $\f$-chain in $S$ eventually stabilises and $T$ is $\arn$.
\end{enumerate}
\end{thm}

\begin{proof}
Let $U=S\s T.$

($\Rightarrow$)  Suppose that (1) does not hold.  Then either $S$ is not $\arn$ or $S$ has an element with a local right identity.  In the former case, $T$ is $\arn$ by Lemma \ref{lem:sdp1}, and in the latter case, $T$ is $\arn$ by Lemma \ref{lem:sdp2}.  

Now suppose for a contradiction that there exists an infinite $\f$-chain
$$a_1\big(\f_{b_1}(S)\big)^1\subsetneq a_2\big(\f_{b_2}(S)\big)^1\subsetneq\cdots$$
in $S.$  Then, for each $i\in\N,$ there exists $t_i\in T$ such that $b_i=b_{i+1}t_i$.  Thus, we have an ascending chain $$b_1\leq_T b_2\leq_T\cdots$$ in $T.$  Since $T$ is $\arn$, there exists $N\in\N$ such that $b_n\,\ar_T\,b_N$ for all $n\geq N.$  Then, by Lemma \ref{lem:f-chain}, we have $\f_{b_n}(S)=\f_{b_N}(S)$ for all $n\geq N.$  Therefore, we have an infinite $\f$-chain
$$a_N\big(\f_{b_N}(S)\big)^1\subsetneq a_{N+1}\big(\f_{b_N}(S)\big)^1\subsetneq\cdots.$$
It follows that for each $n\geq N$ there exists $s_n\in S$ such that $a_n=a_{n+1}\f_{b_N}(s_n).$  Also, since $b_N\,\ar_T\,b_{N+1}$ and $b_N\in b_{N+1}T,$ there exists $t\in T$ such that $b_N=b_Nt.$  Hence, we have $(a_n,b_N)=(a_{n+1},b_N)(s_n,t)$ for each $n\geq N.$   Thus, we have an ascending chain
$$(a_N,b_N)\leq_U(a_{N+1},b_N)\leq_U\cdots$$
in $U.$  Since $U$ is $\arn$, there exists $N'\geq N$ such that $(a_n,b_N)\,\ar_U\,(a_{N'},b_N)$ for all $n\geq N'.$  In particular, there exists $(s,t')\in U$ such that $(a_{N'+1},b_N)=(a_{N'},b_N)(s,t').$  Then $a_{N'+1}=a_{N'}\f_{b_N}(s).$  But then 
$$a_{N'}\big(\f_{b_{N'}}(S)\big)^1=a_{N'}\big(\f_{b_N}(S)\big)^1=a_{N'+1}\big(\f_{b_N}(S)\big)^1=a_{N'+1}\big(\f_{b_{N'+1}}(S)\big)^1,$$ 
and we have a contradiction.

($\Leftarrow$) If (1) holds, then $U$ is $\arn$ by Lemma \ref{lem:sdp3}.  Assume then that (2) holds.  Consider an ascending chain
$$(a_1,b_1)\leq_U(a_2,b_2)\leq_U\cdots$$
in $U.$  We may assume without loss of generality that $(a_i,b_i)\in(a_{i+1},b_{i+1})U$ for each $i\in\N.$  Then, for each $i\in\N,$ there exists $(s_i,t_i)\in U$ such that $(a_i,b_i)=(a_{i+1},b_{i+1})(s_i,t_i).$  Then $a_i=a_{i+1}\f_{b_{i+1}}(s_i)$ and $b_i=b_{i+1}t_i$.  Therefore, we have a $\f$-chain
$$a_1\big(\f_{b_1}(S)\big)^1\subseteq a_2\big(\f_{b_2}(S)\big)^1\subseteq\cdots$$
in $S,$ and an ascending chain $$b_1\leq_Tb_2\leq_T\cdots$$ in $T.$  Since every $\f$-chain eventually stabilises and $T$ is $\arn$, there exists $N\in\N$ such that $a_n\big(\f_{b_n}(S)\big)^1=a_N\big(\f_{b_N}(S)\big)^1$ and $b_n\,\ar_T\,b_N$ for all $n\geq N.$  Therefore, by Lemma \ref{lem:f-chain}, for each $n\geq N$ we have $a_n\in a_N\big(\f_{b_N}(S)\big)^1$ and $\f_{b_n}(S)=\f_{b_N}(S).$  Since $a_N=a_{N+1}\f_{b_{N+1}}(s_N),$ we have
$$a_n\in a_{N+1}\f_{b_{N+1}}(S)\big(\f_{b_N}(S)\big)^1=a_{N+1}\f_{b_{N+1}}(S)\big(\f_{b_{N+1}}(S)\big)^1=a_{N+1}\f_{b_{N+1}}(S).$$  Thus, for each $n\geq N+1,$ there exist $u_n\in S$ and $v_n\in T$ such that $a_n=a_{N+1}\f_{b_{N+1}}(u_n)$ and $b_n=b_{N+1}v_n.$  It follows that $(a_n,b_n)=(a_{N+1},b_{N+1})(u_n,v_n),$ and hence we have $(a_n,b_n)\,\ar_U\,(a_{N+1},b_{N+1}).$  This completes the proof.
\end{proof}

\begin{cor}\label{cor:finite_R-classes}
Let $S$ be an $\arn$ semigroup whose $\ar$-classes are all finite, let $T$ be a semigroup, and let $\f : T\to\emph{End}(S)$ be a homomorphism.  Then $S\s T$ is $\arn$ if and only if either $S$ has no element with a local right identity or $T$ is $\arn$.
\end{cor}

\begin{proof}
By Theorem \ref{thm:sdp}, it suffices to prove that if $T$ is $\arn$ then every $\f$-chain in $S$ eventually stablises.  So, let $T$ be $\arn$ and suppose for a contradiction that there exists an infinite $\f$-chain
$$a_1\big(\f_{b_1}(S)\big)^1\subsetneq a_2\big(\f_{b_2}(S)\big)^1\subsetneq\cdots$$
in $S.$  Then we have ascending chains
$$a_1\leq_Sa_2\leq_S\cdots\quad\text{and}\quad b_1\leq_T b_2\leq_T\cdots$$
in $S$ and $T,$ respectively.  Since $S$ and $T$ are $\arn,$ there exists $N\in\N$ such that $a_n\,\ar_S\,a_N$ and $b_n\,\ar_T\,b_N$ for all $n\geq N.$  Then, by Lemma \ref{lem:f-chain}, we have $\f_{b_n}(S)=\f_{b_N}(S)$ for all $n\geq N.$  Therefore, we have an infinite $\f$-chain
$$a_N\big(\f_{b_N}(S)\big)^1\subsetneq a_{N+1}\big(\f_{b_N}(S)\big)^1\subsetneq\cdots.$$
But then $a_m\neq a_n$ for all $m,n\geq N$ with $m\neq n,$ contradicting the fact that the $\ar_S$-class of $a_N$ is finite.
\end{proof}

Certainly every finite semigroup has an element with a local right identity, so we deduce:

\begin{cor}
Let $S$ be a finite semigroup, let $T$ be a semigroup, and let $\f : T\to\emph{End}(S)$ be a homomorphism.  Then $S\s T$ is $\arn$ if and only if $T$ is $\ar$-noetherian.
\end{cor}

The following example demonstrates that for a semidirect product $S\s T$ to be $\arn$, the semigroup $S$ need not be $\arn$, even in the case that $T$ is trivial.

\begin{ex}
Let $S$ be any monoid with identity 1, let $T=\{e\}$ be the trivial semigroup, and let $\f : T\to\E(S)$ be given by defining $\f_e$ to be the constant map $c_1$ on 1.  Then $(a,e)(a',e)=(a,e)$ for all $a,a'\in S$, so that $S\s T$ is a left zero semigroup, and is hence $\arn$ by Corollary \ref{cor:min_right_ideals}.  (Alternatively, for any $a, a'\in S,$ we clearly have $a\in a'\big(\f_e(S)\big)^1$ if and only if $a=a'$, so certainly every $\f$-chain in $S$ eventually stabilises.  Since $T$ is trivially $\arn$, it then follows from Theorem \ref{thm:sdp} that $S\s T$ is $\ar$-noetherian.)  
\end{ex}

We now show that a semidirect product may {\em not} be $\arn$ even if both its semidirect factors are $\arn$.

\begin{ex}\label{ex:sdp}
Let $S$ be the disjoint union of (a copy of) the free monogenic monoid $\{x\}^{\ast}$ and a set $\{a_i : i\in\Z\}.$  Define a multiplication on $S,$ extending that on $\{x\}^{\ast},$ by
$$a_ix^n=a_{i-n}\quad\text{and}\quad x^na_j=a_ia_j=a_j$$
for all $n\in\N_0$ and $i,j\in\Z$.  It is straightforward to show that $S$ is a monoid under this multiplication.  Clearly $I=\{a_i : i\in\Z\}$ is an ideal of $S$ and a right zero semigroup; in particular, every element of $I$ has a local right identity in $I.$  Since $I$ and $S/I\cong\N^0$ are $\ar$-noetherian, we have that $S$ is $\arn$ by Proposition \ref{prop:ideal,lri}.  Hence $S^0$ is $\arn$.  Now let $T=\{e\}$ be the trivial semigroup, and let $\f : T\to\E(S^0)$ be given by 
$$\f_e(s)=\begin{cases}
s & \text{if }s\in\{x^n : n\in\N_0\},\\
0 & \text{if }s\in\{a_i : i\in\Z\}\cup\{0\}.
\end{cases}$$
For each $i\geq 0,$ we have $a_i=a_{i+1}x=a_{i+1}\f_e(x),$ so that $a_i\big(\f_e(S)\big)^1\subseteq a_{i+1}\big(\f_e(S)\big)^1$.  For any $s\in S^0$ we have 
$$a_i\f_e(s)\in a_i(\{x_n : n\in\N_0\}\cup\{0\})\subseteq\{a_k : k\leq i\}\cup\{0\}.$$
Thus, we have an infinite $\f$-chain
$$a_1\big(\f_e(S)\big)^1\subsetneq a_2\big(\f_e(S)\big)^1\subsetneq\cdots.$$  Hence, by Theorem \ref{thm:sdp}, $S^0\s T$ is not $\arn$.
\end{ex}

For a semigroup $S,$ we denote by $\SE(S)$ the monoid of all surjective endomorphisms of $S.$  We shall use Theorem \ref{thm:sdp} to deduce necessary and sufficient conditions for $S\s T$ to be $\arn$ in the case that $\f_t\in\SE(S)$ for every $t\in T.$  First we prove another lemma.

\begin{lem}\label{lem:sdp4}
Suppose that there exists an element $b\in T$ such that $b$ has a local right identity and $\f_b\in\emph{SEnd}(S)$.  If $S\s T$ is $\arn$, then $S$ is $\arn$.
\end{lem}

\begin{proof}
Suppose for a contradiction that $S$ is not $\arn$.  Then there exists an infinite strictly ascending chain $$a_1<_Sa_2<_S\cdots$$ in $S.$  Then, for each $i\in\N,$ there exists $s_i\in S$ such that $a_i=a_{i+1}s_i$.  Since $\f_b$ is surjective, for each $i\in\N$ there exists $s_i'\in S$ such that $\f_b(s_i')=s_i$.  Let $t'$ be a local right identity of $b,$ so that $b=bt'.$  Then $(a_i,b)=(a_{i+1},b)(s_i',t').$  Thus, letting $U=S\s T,$ we have an ascending chain $$(a_1,b)\leq_U(a_2,b)\leq_U\cdots$$ in $U.$  Since $U$ is $\arn$, there exists $N\in\N$ such that $(a_n,b)\,\ar_U\,(a_N,b)$ for all $n\geq N.$  In particular, there exists $(s,t)\in U$ such that $(a_{N+1},b)=(a_N,b)(s,t).$  But then $a_{N+1}=a_N\f_b(s),$ contradicting that $a_N<_Ta_{N+1}$. 
\end{proof}

\begin{thm}\label{thm:surj}
Let $S$ and $T$ be semigroups, and let $\f : T\to\emph{SEnd}(S)$ be a homomorphism.  Then $S\s T$ is $\arn$ if and only if one of the following holds:
\begin{enumerate}
\item both $S$ and $T$ are $\arn$;
\item $S$ is $\arn$ and has no element with a local right identity;
\item $T$ is $\arn$ and has no element with a local right identity.
\end{enumerate}
\end{thm}

\begin{proof}
($\Rightarrow$) Assume that (2) and (3) do not hold.  Then $T$ has an element, say $b,$ with a local right identity.  Since $\f_b\in\SE(S),$ we have that $S$ is $\arn$ by Lemma \ref{lem:sdp4}.  Also, it follows from Theorem \ref{thm:sdp} that $T$ is $\arn$.  Thus (1) holds.

($\Leftarrow$) If (2) or (3) holds, then $S\s T$ is $\arn$ by Lemma \ref{lem:sdp3}.  Assume then that (1) holds.  Since $\f_b(S)=S$ for all $b\in T,$ every $\f$-chain in $S$ is an ascending chain of principal right ideals of $S,$ and hence must eventually stabilise as $S$ is $\arn.$  Therefore, since $T$ is $\arn$, we have that $S\s T$ is $\arn$ by Theorem \ref{thm:sdp}.
\end{proof}

\begin{cor}\cite[Theorem 3.13]{Stopar}
Let $S$ and $T$ be semigroups with idempotents, and let $\f : T\to\emph{SEnd}(S)$ be a homomorphism.  Then $S\s T$ is $\arn$ if and only if $S$ and $T$ are $\arn$.
\end{cor}

\begin{cor}\label{cor:dp}
Let $S$ and $T$ be semigroups.  Then the direct product $S\times T$ is $\arn$ if and only if one of the following holds:
\begin{enumerate}
\item both $S$ and $T$ are $\arn$;
\item $S$ is $\arn$ and has no element with a local right identity;
\item $T$ is $\arn$ and has no element with a local right identity.
\end{enumerate}
\end{cor}

\section{\large{Sch{\"u}tzenberger Products}\nopunct}
\label{sec:sp}

The Sch{\"u}tzenberger product of semigroups was introduced by Sch{\"u}tzenberger in \cite{Schutz} in relation to the study of finite aperiodic monoids.  It has since found many other useful applications in semigroup theory; see \cite{Margolis,Pin} for instance.

For any set $X,$ let $\P_f(X)$ denote the set of all finite subsets of $X.$  Let $S$ and $T$ be semigroups.  For $s\in S,$ $t\in T$ and $P\in\P_f(S\times T),$ we define
$$sP=\{(sp, q) : (p, q)\in P\}\quad\text{and}\quad Pt=\{(p, qt) : (p, q)\in P\}.$$
The {\em Sch{\"u}tzenberger product} of $S$ and $T,$ denoted by $S\D T,$ is the semigroup with universe $S\times\P_f(S\times T)\times T$ and multiplication given by
$$(s_1, P_1, t_1)(s_2, P_2, t_2)=(s_1s_2, s_1P_2\cup P_1t_2, t_1t_2).$$
Observe that the direct product $S\times T$ embeds into $S\D T$ via $(s,t)\mapsto(s,\emptyset,t)$.  Unlike for $S\times T,$ the multiplication in $S\D T$ is assymmetrical, and hence $S\Diamond T$ is not in general isomorphic to $T\Diamond S$.
The main theme of this section is the relationship between the Sch{\"u}tzenberger product and the direct product with regard to being $\arn$.  We begin with the following lemma.

\begin{lem}
\label{lem:sp}
Let $S$ and $T$ be semigroups.  If $S\D T$ is $\arn$, then so is the direct product $S\times T.$
\end{lem}

\begin{proof}
Notice that $S\times\{\emptyset\}\times T\cong S\times T$ and that $S\times\{\emptyset\}\times T$ is a subsemigroup of $S\D T$ such that $(S\D T)\!\setminus\!(S\times\{\emptyset\}\times T)$ is an ideal of $S\D T.$  Therefore, by Corollary \ref{cor:comp_left_ideal}, if $S\D T$ is $\arn$ then is $S\times T.$
\end{proof}

Letting $\{1\}$ denote the trivial group, it is clear that $\{1\}\D T$ is isomorphic to the semigroup $\P_f(T)\times T$ with multiplication given by
$$(P_1,t_1)(P_2,t_2)=(P_1t_2\cup P_2,t_1t_2),$$ where $Pt=\{pt : p\in P\}$ for $P\in\P_f(T)$ and $t\in T.$

\begin{lem}
Let $S$ and $T$ be semigroups, and suppose that there exists an element $a\in S$ that has a local right identity in $S.$  If $S\D T$ is $\arn,$ then $\{1\}\D T$ is $\arn.$
\end{lem}

\begin{proof}
We prove the contrapositive.  Let $U=\P_f(T)\times T\cong\{1\}\D T$ and let $V=S\D T.$  Assume that $U$ is not $\arn.$  Then there exists an infinite strictly ascending chain
$$(P_1,b_1)<_U(P_2,b_2)<_U\cdots$$
in $U.$  Then, for each $i\in\N,$ there exists $(Q_i,t_i)\in U$ such that $(P_i,b_i)=(P_{i+1},b_{i+1})(Q_i,t_i).$  Let $s$ be a local right identity of $a$ (so $as=a$), and for each $i\in\N$ let $P_i'=\{a\}\times P_i$ and $Q_i'=\{s\}\times Q_i$.  We then have
$$(a,P_i',b_i)=(a,P_{i+1}',b_{i+1})(s,Q_i',t_i).$$
Suppose for a contradiction that there exists $i\in\N$ such that $(a,P_{i+1}',b_{i+1})\in(a,P_i',b_i)V^1.$  Then, either $(a,P_{i+1}',b_{i+1})=(a,P_i',b_i)$ or there exists $(s',Q,t')\in V$ such that 
$$(a,P_{i+1}',b_{i+1})=(a,P_i',b_i)(s',Q,t')=(as',P_i't'\cup aQ, b_it').$$
But then either $(P_{i+1},b_{i+1})=(P_i,b_i)$ or, letting $\pi_T$ denote the projection map $S\times T\to T,$ we have 
$$(P_{i+1},b_{i+1})=(P_i,b_i)(Q\pi_T,t')$$
so that $(P_{i+1},b_{i+1})\in(P_i,b_i)U^1$, a contradiction.  Thus, we have an infinite strictly ascending chain
$$(a,P_1',b_1)<_V(a,P_2',b_2)<_V\cdots$$
in $V.$  Hence $V$ is not $\arn.$
\end{proof}

We now provide an example of an $\arn$ semigroup $T$ such that $\{1\}\D T$ is not $\arn.$  In particular, the converse of Lemma \ref{lem:sp} does not hold (since $\{1\}\times T\cong T$).

\begin{ex}
Let $T$ be the monoid $S$ from Example \ref{ex:sdp}; that is, $$T=\{x^n : n\in\N_0\}\sqcup\{a_i : i\in\Z\}$$ with product given by
$$x^mx^n=x^{m+n},\quad a_ix^n=a_{i-n}\quad\text{and}\quad x^na_j=a_ia_j=a_j.$$
As shown in Example \ref{ex:sdp}, $T$ is $\arn$.
We prove that $U=\P_f(T)\times T\cong\{1\}\D T$ is not $\ar$-noetherian.

Consider $i\in\N.$  We have $(\{a_i\},a_{i+1})=(\{a_{i+1}\},a_{i+2})(\emptyset,x).$
Suppose for a contradiction that $(\{a_{i+1}\},a_{i+2})=(\{a_i\},a_{i+1})(P,t)$ for some $(P,t)\in U.$  Then $a_{i+1}=a_it$ and $a_{i+2}=a_{i+1}t.$  We cannot have $t\in\{x\}^{\ast}$, for then we would have $a_it\in\{a_j : j\leq i\},$ and we cannot have $t\in\{a_j : j\in\Z\},$ for then we would have $a_it=a_{i+1}t=t.$  Thus, no such $t$ exists, and we have the desired contradiction.  It follows that we have an infinite ascending chain
$$(\{a_1\},a_2)<_U(\{a_2\},a_3)<_U\cdots$$
in $U,$ so $U$ is not $\ar$-noetherian.
\end{ex}

In the remainder of this section, we explore situations in which $S\D T$ being $\arn$ is equivalent to $S\times T$ being $\arn.$  This turns out to be the case when $T$ is finite or cancellative.

Since $S$ and $T$ are homomorphic images of $S\D T,$ by Proposition \ref{prop:nolri} we have:

\begin{lem}
\label{lem:sp,nolri}
Let $S$ and $T$ be semigroups.  If either of $S$ and $T$ is $\arn$ and has no element with a local right identity, then $S\D T$ is $\arn$.
\end{lem}


Lemmas \ref{lem:sp} and \ref{lem:sp,nolri} and Corollary \ref{cor:dp} yield:

\begin{prop}
Let $S$ and $T$ be semigroups where $S$ (resp.\ T) is not $\ar$-noetherian.  Then the following are equivalent:
\begin{enumerate}
\item $S\D T$ is $\arn$;
\item $S\times T$ is $\arn$;
\item $T$ (resp.\ $S$) is $\arn$ and has no element with a local right identity.
\end{enumerate}
\end{prop}


\begin{thm}
Let $S$ and $T$ be semigroups where $T$ is finite.  Then the following are equivalent:
\begin{enumerate}
\item $S\D T$ is $\arn$;
\item $S\times T$ is $\arn$;
\item $S$ is $\arn$.
\end{enumerate}
\end{thm}

\begin{proof}
(1)$\Rightarrow$(2) is Lemma \ref{lem:sp}, and (2)$\Rightarrow$(3) follows from Corollary \ref{cor:dp}.

(3)$\Rightarrow$(1).  Let $U=S\D T,$ and suppose for a contradiction that $U$ is not $\arn.$  Then there exists an infinite strictly ascending chain
$$(a_1, P_1, b_1)<_U(a_2, P_2, b_2)<_U\cdots$$
in $U.$  For each $i\in\N,$ there exists $(s_i, Q_i, t_i)\in U$ such that 
$$(a_i, P_i, b_i)=(a_{i+1}, P_{i+1}, b_{i+1})(s_i, Q_i, t_i)=(a_{i+1}s_i, P_{i+1}t_i\cup a_{i+1}Q_i, b_{i+1}t_i),$$
so $a_i=a_{i+1}s_i$, $P_i=P_{i+1}t_i\cup a_{i+1}Q_i$ and $b_i=b_{i+1}t_i$.  Thus, we have an ascending chain $$a_1\leq_Sa_2\leq_S\cdots$$ in $S.$  Observe that for each $i\in\N$ we have 
$$P_i\pi_S=P_{i+1}\pi_S\cup a_{i+1}Q_i\pi_S.$$
It follows that $$P_1\pi_S\supseteq P_2\pi_S\supseteq\cdots.$$
Since $S$ is $\arn$, and there does not exist any infinite strictly descending chain of finite subsets of $S,$ there exists $N\in\N$ such that $a_n\,\ar_S\,a_N$ and $P_n\pi_S=P_N\pi_S$ for all $n\geq N.$  Then, for each $n\geq N,$ we have $P_n\subseteq P_N\pi_S\times T.$  Since $P_N$ and $T$ are finite (and hence $P_N\pi_S\times T$ is finite), it follows that there exist $m, n\geq N$ with $m\leq n-2$ such that $P_m=P_n$ and $b_m=b_n$.  Since $a_m\,\ar_S\,a_n$ and $a_m\in a_nS,$ there exists $s'\in S$ such that $a_n=a_ms'.$  Then we have
\begin{align*}
(a_{m+1},P_{m+1},b_{m+1})(s_ms',Q_m,t_m)&=(a_{m+1}s_ms',P_{m+1}t_m\cup a_{m+1}Q_m,b_{m+1}t_m)\\&=(a_ms',P_m,b_m)\\
&=(a_n,P_n,b_n).
\end{align*}
But this contradicts that $(a_{m+1},P_{m+1},b_{m+1})<_U(a_n,P_n,b_n).$  Hence $U$ is $\arn.$
\end{proof}

\begin{thm}
\label{thm:sp}
Let $S$ and $T$ be semigroups where $T$ is cancellative.  Then $S\D T$ is $\arn$ if and only if $S\times T$ is $\arn$.
\end{thm}

\begin{proof}
($\Rightarrow$).  This follows immediately from Lemma \ref{lem:sp}.

($\Leftarrow$)  Let $U=S\D T,$ and suppose for a contradiction that $U$ is not $\arn$.  Then, by Lemma \ref{lem:sp,nolri} and Corollary \ref{cor:dp}, both $S$ and $T$ are $\arn$ and have elements with local right identities.  It is straightforward to show that a cancellative semigroup has an element with a local right identity if and only it is a monoid; thus $T$ is a monoid.

Now, there exists an infinite strictly ascending chain
$$(a_1, P_1, b_1)<_U(a_2, P_2, b_2)<_U\cdots$$
in $U.$  For each $i\in\N,$ there exists $(s_i, Q_i, t_i)\in U$ such that 
$$(a_i, P_i, b_i)=(a_{i+1}, P_{i+1}, b_{i+1})(s_i, Q_i, t_i)=(a_{i+1}s_i, P_{i+1}t_i\cup a_{i+1}Q_i, b_{i+1}t_i),$$
so $a_i=a_{i+1}s_i$, $P_i=P_{i+1}t_i\cup a_{i+1}Q_i$ and $b_i=b_{i+1}t_i$.
Thus, we have ascending chains
$$a_1\leq_Sa_2\leq_S\cdots\;\;\text{ and }\;\;b_1\leq_Tb_2\leq_T\cdots$$
in $S$ and $T,$ respectively.  Since $S$ is $\arn$, there exists some $N\in\N$ such that $a_n\,\ar_S\,a_N$ for all $n\geq N.$

Now, it follows from the cancellativity of $T$ that the maps
$$P_{i+1}\to P_{i+1}t_i, (x,y)\mapsto (x,yt_i)\;(i\in\mathbb{N})$$
are bijections.  Using the fact that $P_{i+1}t_i\subseteq P_i$, we deduce that we have a chain $$|P_1|\geq|P_2|\geq\cdots$$ of non-negative integers.  This chain must eventually stabilise; we may assume without loss of generality that $|P_n|=|P_N|$ for all $n\geq N.$  Then $P_n=P_{n+1}t_n$ for all $n\geq N.$  In particular, we have 
$$P_N\pi_S=P_{N+1}\pi_S=\cdots.$$
Let $|P_N|=m,$ and for all $n\geq N$ let
$$P_n=\{(u_1, v_{n,1}), \dots, (u_m, v_{n,m})\}$$
such that $v_{n,j}=v_{n+1,j}t_n$ for each $j\in\{1, \dots, m\}.$
Then we have ascending chains
$$\begin{matrix}
b_N \!&\! \leq_T \!&\! b_{N+1} \!&\! \leq_T \!&\! \cdots\\
v_{N,1} \!&\! \leq_T \!&\! v_{N+1,1} \!&\! \leq_T \!&\! \cdots\\
v_{N,2} \!&\! \leq_T \!&\! v_{N+1,2} \!&\! \leq_T \!&\! \cdots\\
& & \vdots & &\\
v_{N,m} \!&\! \leq_T \!&\! v_{N+1,m} \!&\! \leq_T \!&\! \cdots  
\end{matrix}$$
Since $T$ is $\arn$, the above chains must all eventually stabilise.  So, there exists $N'\geq N$ such that, for all $n\geq N',$ we have $b_n\,\ar_T\,b_{N'}$ and $v_{n,j}\,\ar_T\,v_{N',j}$ for each $j\in\{1, \dots, m\}.$
Consider $n\geq N'.$  There exists $s_n'\in S$ such that $a_{n+1}=a_ns_n'.$  Also, there exists $x_n\in T$ such that $b_{n+1}=b_nx_n,$ and for each $j\in\{1, \dots, m\}$ there exists $x_{n,j}\in T$ such that $v_{n+1,j}=v_{n,j}x_{n,j}$.  Then $b_n=b_nx_nt_n$, and $v_{n,j}=v_{n,j}x_{n,j}t_n$ for each $j\in\{1, \dots, m\}.$  Since $T$ is cancellative, it follows that $x_nt_n=1$ and $x_{n,j}t_n=1$ for each $j\in\{1, \dots, m\}.$  By cancellativity again we have
$$x_n=x_{n,1}=\dots=x_{n,m}.$$
It follows that $(u_{n,j},v_{n+1,j})=(u_{n,j},v_{n,j}x_n)$
for each $j\in\{1, \dots, m\},$ and hence $P_{n+1}=P_nx_n$.
Thus, we have
$$(a_{n+1}, P_{n+1}, b_{n+1})=(a_n, P_n, b_n)(s_n', \emptyset, x_n).$$
But this contradicts the assumption that $(a_n, P_n, b_n)<_U(a_{n+1}, P_{n+1}, b_{n+1}).$  Hence $U$ is $\arn$.
\end{proof} 

By Theorem \ref{thm:sp} and Corollary \ref{cor:dp}, we have:

\begin{cor}
\label{cor:sp,group}
Let $S$ be a semigroup and let $G$ be a group.  Then the following are equivalent:
\begin{enumerate}
\item $S\D G$ is $\arn$;
\item $S\times G$ is $\arn$;
\item $S$ is $\arn$.
\end{enumerate}
\end{cor}

\section{\large{Free Products}\nopunct}
\label{sec:fp}


Let $S_i$ ($i\in I$) be a collection of pairwise disjoint semigroups.  Let $S$ be the set of all finite non-empty sequences $(a_1,\dots,a_m)$ where $a_j\in\bigcup_{i\in I}S_i$ ($1\leq j\leq m$) and each $a_k$ belongs to a different $S_i$ to that of $a_{k+1}$ ($1\leq k\leq m-1$).
Define a multiplication on $S$ as follows:
$$(a_1,\dots,a_m)(b_1,\dots,b_n)=\begin{cases}
(a_1,\dots,a_m,b_1,\dots,b_n)&\text{if }a_m\in S_i, b_1\in S_j\text{ where }i\neq j,\\
(a_1,\dots,a_mb_1,\dots,b_n)&\text{if }a_m, b_1\in S_i\text{ for some }i\in I.
\end{cases}$$
It is straightforward to verify that this multiplication is associative.  The semigroup $S$ under this multiplication is called the {\em semigroup free product} of $S_i$ ($i\in I$) and is denoted by $\prod\!\!\!\!~^{\scaleto{*}{6pt}}\{S_i : i\in I\}.$

Now suppose that the semigroups $S_i$ ($i\in I$) are monoids with identities $1_i$, respectively.  Let $\rho$ be the congruence on $\prod\!\!\!\!~^{\scaleto{*}{6pt}}\{S_i : i\in I\}$ generated by 
$$\{(1_i,1_j) : i, j\in I, i\neq j\},$$
and denote the $\rho$-class $\{1_i : i\in I\}$ by 1.
The {\em monoid free product} of $S_i$ ($i\in I$), denoted by $\prod\!\!\!\!~^{\scaleto{*}{6pt}}_1\{S_i : i\in I\},$ is the monoid $\prod\!\!\!\!~^{\scaleto{*}{6pt}}\{S_i : i\in I\}/\rho$ with identity 1. 

We note that the monoid free product of groups coincides with the group free product \cite[p.\ 266]{Howie}.

The following result provides necessary and sufficient conditions for a semigroup free product to be $\arn$.

\begin{thm}
\label{thm:sfp}
Let $S_i$ ($i\in I$) be a collection of pairwise disjoint semigroups.  Then the semigroup free product $\prod\!\!\!\!~^{\scaleto{*}{6pt}}\{S_i : i\in I\}$ is $\arn$ if and only if each $S_i$ ($i\in I$) is $\arn$.
\end{thm}

\begin{proof} 
Let $S=\prod\!\!\!\!~^{\scaleto{*}{6pt}}\{S_i : i\in I\}.$  Notice that each $S_i$ embeds into $S$ via $a\mapsto(a)$; we shall identify $S_i$ with its image under this mapping.  We denote the `length' of $u\in S$ by $|u|$, i.e.\ if $u=(a_1,\dots,a_m)$ then $|u|=m.$  Oberve that for any $u, v\in S$ we have $|uv|\in\{|u|+|v|-1,|u|+|v|\}.$

($\Rightarrow$) We claim that each $S_i$ is an $\ar$-preserving subsemigroup of $S,$ and is hence $\arn$ by Proposition \ref{prop:R-preserving}.  Indeed, let $a\leq_Sb$ where $a, b\in S_i$.  Either $a=b$ or $a=bs$ for some $s\in S.$  In the latter case, we must have $s\in S_i$, for otherwise $|bs|=|b|+|s|>1=|a|.$  Thus $a\leq_{S_i}b,$ as required.

($\Leftarrow$) Consider an ascending chain 
$$u_1\leq_Su_2\leq_S\cdots$$
in $S.$  Then, for each $n\in\N,$ there exists $v_n\in S^1$ such that $u_n=u_{n+1}v_n$.  Then $$|u_1|\geq|u_2|\geq\cdots.$$ 
Hence, there exists $N\in\N$ such that $|u_n|=|u_N|$ for all $n\geq N.$  Let $m=|U_N|.$  It follows from the definition of the multiplication in $S$ that the $u_n$ $(n\geq N)$ have the same first $m-1$ terms, and that there exists $i\in I$ such that the $m$-th term of each $u_n$ belongs to $S_i$ and $v_n\in S_i^1$ ($n\geq N$).  Thus, for each $n\geq N,$ we let $u_n=(a_1\dots,a_{m-1},b_n)$ where $b_n\in S_i$.  Then $b_n=b_{n+1}v_n\in b_{n+1}S_i^1$.  Thus, we have an ascending chain
$$b_N\leq_{S_i}b_{N+1}\leq_{S_i}\cdots$$
in $S_i$.  Since $S_i$ is $\arn$, there exists $N'\geq N$ such that $b_n\,\ar_{S_i}\,b_{N'}$ for all $n\geq N'.$  Therefore, for each $n\geq N$ there exists $s_n\in S_i^1$ such that $b_n=b_{N'}s_n$.  If $s_n=1,$ then $b_n=b_{N'}$ and hence $u_n=u_{N'}$.  Otherwise, if $s_n\in S_i$, we have
$$u_n=(a_1\dots,a_{m-1},b_{N'}s_n)=(a_1\dots,a_{m-1},b_{N'})(s_n)=u_{N'}(s_n).$$
Thus, we have $u_n\,\ar_S\,u_{N'}$ for all $n\geq N'.$  Hence $S$ is $\arn.$
\end{proof}

We now turn our attention to the monoid free product.  First, we make some observations regarding this construction.

Consider a monoid free product $S=\prod\!\!\!\!~^{\scaleto{*}{6pt}}_1\{S_i : i\in I\}.$  We may view the non-identity elements of $\prod\!\!\!\!~^{\scaleto{*}{6pt}}_1\{S_i : i\in I\}$ as sequences $(a_1,\dots,a_n)\in\prod\!\!\!\!~^{\scaleto{*}{6pt}}\{S_i : i\in I\}$ where each $a_i$ belongs to some $S_i\!\setminus\!\{1_i\}$ \cite[p.\ 266]{Howie}.  More precisely, with $\rho$ as given above, in each non-identity $\rho$-class there exists a unique sequence that contains no elements from $\{1_i : i\in I\}$; we call this sequence {\em reduced}.  Thus, we identify the non-identity elements of $S$ with their corresponding reduced sequences.  

Now, consider a reduced squence $u=(a_1,\cdots,a_n)\in S.$  Letting $a_n\in S_i$, observe that if $a_n$ is not right invertible in $S_i$, then for any $v\in S\!\setminus\!\{1\}$ we have $|uv|\in\{|u|+|v|-1,|u|+|v|\}.$  It follows that, if $a_n$ is not right invertible in $S_i$, $u$ is of minimal length in its $\ar$-class (i.e. $|u|=\min\{|w| : u\,\ar\,w\}$).  In fact, the converse also holds.  Indeed, if $a_n$ is right invertible, then there exists $s\in S_i$ such that $a_ns=1_i$.  Then $us=(a_1,\dots,a_{n-1}),$ and of course $(a_1,\dots,a_{n-1})(a_n)=u,$ so $a_n\,\ar\,(a_1,\dots,a_{n-1}).$  Hence, $u$ is not of minimal length in its $\ar$-class.  

\begin{thm}
\label{thm:mfp}
Let $S_i$ ($i\in I$) be a collection of pairwise disjoint monoids.  Then the monoid free product $\prod\!\!\!\!~^{\scaleto{*}{6pt}}_1\{S_i : i\in I\}$ is $\arn$ if and only if each $S_i$ ($i\in I$) is $\arn$.
\end{thm}

\begin{proof}
The proof is essentially the same as that of Theorem \ref{thm:sfp}.  The only difference is that, in ($\Leftarrow$), we stipulate that each $u_i$ is of minimal length in its $\ar$-class, and it then follows that $|u_1|\geq|u_2|\geq\cdots.$
\end{proof}

\section{\large{Rees Matrix Semigroups}\nopunct}
\label{sec:Rees}

Let $S$ be a semigroup, let $I$ and $J$ be non-empty index sets, and let $P=(p_{ji})$ be a $J\times I$ matrix with entries from $S.$  The set $I\times S\times J$ becomes a semigroup under the multiplication given by 
$$(i, s, j)(k, t,l)=(i, sp_{jk}t, l),$$
and is called the {\em Rees matrix semigroup over $S$ with respect to $P$}.  We denote this semigroup by $\M(S; I, J; P).$

We note that Rees matrix semigroups over groups are precisely the completely simple semigroups (i.e.\ semigroups with no proper ideals that possess minimal left and right ideals) \cite[Theorem 3.3.1]{Howie}.

We now state the main result of this section, providing necessary and sufficient conditions for a Rees matrix semigroup to be $\arn.$

\begin{thm}
\label{thm:rm}
Let $T=\M(S; I, J; P)$ be a Rees matrix semigroup.  Let $U$ denote the right ideal $\{p_{j,i} : j\in J, i\in I\}S$ of $S.$  Then $T$ is $\arn$ if and only if every ascending chain
$$a_1U^1\subseteq a_2U^1\subseteq\cdots,$$
where $a_i\in S,$ eventually stabilises.
\end{thm}

\begin{proof}
We prove the contrapositive for both directions.

($\Rightarrow$) Suppose that there exists an infinite strictly ascending chain
$$a_1U^1\subsetneq a_2U^1\subsetneq\cdots,$$
where $a_i\in S.$  Fix $i\in I$ and $j_1\in J.$  For each $n\in\N$ there exists $i_{n+1}\in I,$ $j_{n+1}\in J$ and $s_n\in S$ such that $a_n=a_{n+1}p_{j_{n+1},i_{n+1}}s_n.$
Then $$(i, a_n, j_n)=(i, a_{n+1}, j_{n+1})(i_{n+1}, s_n, j_n),$$
so $(i, a_n, j_n)\leq_T(i, a_{n+1}, j_{n+1}).$
Suppose for a contradiction that $(i, a_n, j_n)\,\ar_T\,(i, a_{n+1}, j_{n+1}).$  We cannot have $(i, a_n, j_n)=(i, a_{n+1}, j_{n+1}),$ for then $a_n=a_{n+1}$.  Therefore, there exist some $k\in I$ and $s\in S$ such that $(i, a_{n+1}, j_{n+1})=(i, a_n, j_n)(k, s, j_{n+1}).$  But then $a_{n+1}=a_n(p_{j_n,k}s)\in a_nU,$ so that $a_nU^1=a_{n+1}U^1$, a contradiction. Thus, we have an infinite strictly ascending chain
$$(i, a_1, j_1)<_T(i, a_2, j_2)<_T\cdots$$
in $T,$ and hence $T$ is not $\arn$.

($\Leftarrow$)  Suppose that $T$ is not $\arn$.  Then there exists an infinite strictly ascending chain
$$(i_1, a_1, j_1)<_T(i_2, a_2, j_2)<_T\cdots$$
in $T.$ Letting $i=i_1,$ we have $$i=i_1=i_2=\cdots.$$  For each $n\in\N$ there exist $k_n\in I$ and $s_n\in S$ such that 
$$(i, a_n, j_n)=(i, a_{n+1}, j_{n+1})(k_n, s_n, j_n).$$  Thus $a_n=a_{n+1}p_{j_{n+1},k_n}s_n\in a_{n+1}U,$ so $a_nU^1\subseteq a_{n+1}U^1.$  We cannot have $a_nU^1=a_{n+2}U^1.$  Indeed, if we did, then there would exist $u\in U^1$ such that $a_{n+2}=a_nu,$ and hence $a_{n+2}=a_{n+1}(p_{j_{n+1},k_n}s_nu).$  But then
$$(i, a_{n+2}, j_{n+2})=(i, a_{n+1}, j_{n+1})(k_n, s_nu, j_{n+2})\in(i, a_{n+1}, j_{n+1})T,$$ contradicting the fact that $(i, a_{n+1}, j_{n+1})<_T(i, a_{n+2}, j_{n+2}).$  Thus, we have an infinite strictly ascending chain
$$a_1U^1\subsetneq a_3U^1\subsetneq a_5U^1\subsetneq\cdots,$$
as desired.
\end{proof}

\begin{cor}
\label{cor:rm}
Let $T=\M(S; I, J; P)$ be a Rees matrix semigroup.  If $S$ is $\arn$ then so is $T.$
\end{cor}

\begin{proof}
Let $U$ be as given in the statement of Theorem \ref{thm:rm}.  Suppose for a contradiction that there exists an infinite strictly ascending chain 
$$a_1U^1\subsetneq a_2U^1\subsetneq\cdots$$
where $a_i\in S.$  Then clearly we have an ascending chain 
$$a_1\leq_Sa_2\leq_S\cdots.$$  
Since $S$ is $\arn$, there exists $N\in\N$ such that $a_nS^1=a_NS^1$ for all $n\geq N.$  But then
$$a_{N+2}\in a_NS\subseteq a_{N+1}US\subseteq a_{N+1}U,$$
contradicting the fact that $a_{N+1}U^1\subsetneq a_{N+2}U^1$.
Hence, by Theorem \ref{thm:rm}, $T$ is $\arn$.
\end{proof}

The converse of Corollary \ref{cor:rm} does not hold, as demonstrated by the following example.

\begin{ex}
Let $S$ be a semigroup with 0 that is not $\arn$.  Let $P$ be the $1\times 1$ matrix whose entry is 0, and let $T=\M(S; \{1\}, \{1\}; P).$  For any $s, s'\in S$ we have $(1, s, 1)(1, s', 1)=(1, s0s', 1)=(1, 0, 1),$ and clearly $(1, 0, 1)$ is a zero element in $T,$ so $T$ is a null semigroup.  Hence, by Corollary \ref{cor:min_right_ideals}, $T$ is $\arn$.  
\end{ex}

\begin{cor}
\label{cor:rm1}
Let $T=\M(S; I, J; P)$ be a Rees matrix semigroup such that every element of $S$ has a local right identity in $U=\{p_{j,i} : j\in J, i\in I\}S.$  Then $T$ is $\arn$ if and only if $S$ is $\arn$.
\end{cor}

\begin{proof}
Consider any $a\in S.$  By assumption, we have $a\in aU,$ so $aS^1\subseteq aUS^1\subseteq aU^1.$
Clearly $aU^1\subseteq aS^1,$ so $aS^1=aU^1$.
The result now follows readily from Theorem \ref{thm:rm}.
\end{proof}

\begin{cor}
\label{cor:rm2}
Let $T=\M(S; I, J; P)$ be a Rees matrix semigroup where $S$ is a monoid, and suppose that there exist $i\in I$ and $j\in J$ such that $p_{j,i}$ is right invertible.  Then $T$ is $\arn$ if and only if $S$ is $\arn$.
\end{cor}

\begin{proof}
We have $1_S\in p_{j,i}S,$ and $1_S$ is obviously a local right identity of every element  of $S.$  Hence, by Corollary \ref{cor:rm1}, $S$ is $\arn$.
\end{proof}

We now consider a variant of the Rees matrix construction.  Let $S$ be a semigroup with zero 0, let $I$ and $J$ be non-empty index sets, and let $P=(p_{ji})$ be a $J\times I$ matrix with entries from $S.$  Let $T'=\M(S; I, J; P),$ and let $T$ denote the Rees quotient $T'/Q,$ where $Q$ is the ideal $I\times\{0\}\times J$ of $T'$.  The semigroup $T$ is called the {\em Rees matrix semigroup with zero over $S$ with respect to $P$}, and is denoted by $\M^0(S; I, J; P).$

Rees matrix semigroups with zero over groups are precisely the completely 0-simple semigroups \cite[Theorem 3.2.3]{Howie}.

\begin{cor}\label{cor:rm0}
Let $T=\M^0(S; I, J; P)$ be a Rees matrix semigroup with zero.  Let $U$ denote the right ideal $\{p_{j,i} : j\in J, i\in I\}S$ of $S.$  Then the following are equivalent:
\begin{enumerate}
\item $T$ is $\arn$;
\item $\M(S; I, J; P)$ is $\arn$;
\item every ascending chain
$$a_1U^1\subseteq a_2U^1\subseteq\cdots,$$
where $a_i\in S,$ eventually stabilises.
\end{enumerate}
\end{cor}

\begin{proof}
(1)$\Leftrightarrow$(2).  Let $T'=\M(S; I, J; P)$ and $Q=I\times\{0\}\times J,$ so that $T=T'/Q.$  Since $Q$ is $\arn$ and every element of $Q$ has a local right identity in $Q,$ it follows from Proposition \ref{prop:ideal,lri} that $T$ is $\arn$ if and only if $T'$ is $\arn$.

(2) and (3) are equivalent by Theorem \ref{thm:rm}.
\end{proof}

Related to the Rees matrix with zero construction is that of the Brandt extension, defined as follows.  Let $S$ be a semigroup and let $I$ be a non-empty set.  The {\em Brandt extension} of $S$ by $I,$ denote by $\mathcal{B}(S, I),$ is the semigroup with universe $(I\times S\times I)\cup\{0\}$ and multiplication given by $$(i, s, j)(k, t,l)=
\begin{cases}
(i, st, l)&\text{if }j=k\\
0&\text{ otherwise,}
\end{cases}$$ 
and $0x=x0=0$ for all $x\in(I\times S\times I)\cup\{0\}.$
Notice that if $S$ is a monoid, then $\mathcal{B}(S, I)$ is isomorphic to $\M^0(S; I, I; P)$ where $P$ is the $I\times I$ identity matrix.  Brandt extensions of groups are precisely the completely $0$-simple inverse semigroups \cite[Theorem 5.1.8]{Howie}.

\begin{thm}\label{thm:Brandt}
Let $S$ be a semigroup and let $I$ be a non-empty set.  Then $\mathcal{B}(S,I)$ is $\arn$ if and only if $S$ is $\arn$.
\end{thm}

\begin{proof}
($\Rightarrow$) It is straightforward to show that, for any $i\in I,$ $S$ is isomorphic the subsemigroup $S_i=\{i\}\times S\times\{i\}$ of $\mathcal{B}(S,I),$ and that $S_i$ is right unitary in $\mathcal{B}(S,I).$  Hence, $S$ is $\arn$ by Corollary \ref{cor:ru}.

($\Leftarrow$) Letting $T=\mathcal{B}(S^1,I),$ we have $T\cong\M^0(S^1; I, I; P)$ where $P$ is the $I\times I$ identity matrix.  Since $S$ is $\arn$, we have that $\M(S^1; I, I; P)$ is $\arn$ by Corollary \ref{cor:rm}.  Hence, by Corollary \ref{cor:rm0}, $T$ is $\arn$.  Since $\mathcal{B}(S,I)$ is an ideal of $T,$ it is also $\arn$ by Proposition \ref{prop:ideal}.
\end{proof}

\section{\large{Bruck-Reilly Extensions}\nopunct}
\label{sec:BR}

Let $M$ be a monoid with identity $1_M$, and let $\theta : M\to M$ be an endomorphism.  We define a binary operation on the set $\N_0\times M\times\N_0$ by
$$(i, a, j)(p, b, q)=(i-j+t, (a\theta^{t-j})(b\theta^{t-p}), q-p+t),$$
where $t=\max(j, p)$ and $\theta^0$ denotes the identity map on $M.$
With this operation the set $\N_0\times M\times\N_0$ is a monoid with identity $(0, 1_M, 0).$  It is denoted by $BR(M, \theta)$ and called the {\em Bruck-Reilly extension of $M$ determined by $\theta$}.

Special instances of this construction were introduced by Bruck \cite{Bruck} and Reilly \cite{Reilly}, after whom it is named, and it was given in its general form by Munn in \cite{Munn}. 

\begin{thm}
\label{thm:BR}
Let $M$ be a monoid and let $\theta : M\to M$ be a monoid homomorphism.  Then $BR(M,\theta)$ is $\arn$ if and only if $M$ is $\arn$.
\end{thm}

\begin{proof}
($\Rightarrow$) Let $N=BR(M, \theta).$  It is straightforward to show that $M$ is isomorphic to the submonoid $\{0\}\times M\times\{0\}$ of $N,$ and that this submonoid is right unitary in $N.$  Hence, $M$ is $\arn$ by Corollary \ref{cor:ru}.

($\Leftarrow$) Consider an ascending chain 
$$u_1\leq_Nu_2\leq_N\cdots$$
in $N,$ where $u_k=(i_k, a_k, j_k).$  Then for each $k\in\N$ there exists $(p_k, m_k, q_k)$ such that $u_k=u_{k+1}(p_k, m_k, q_k).$  Letting $t_k=\max(j_{k+1}, p_k),$ we have 
$$i_k=i_{k+1}-j_{k+1}+t_k.$$  Since $t_k\geq j_{k+1},$ it follows that $i_k\geq i_{k+1}.$  Thus we have 
$$i_1\geq i_2\geq\cdots,$$ and hence there exists $N\in\N$ such that $i_N=i_{N+1}=\cdots.$  Let $i=i_N.$  Then for $k\geq N$ we have $i=i-j_{k+1}+t_k,$ so $t_k=j_{k+1}.$  It follows that, for each $k\geq N,$ we have
$$a_k=a_{k+1}(m_k\theta^{j_{k+1}-p_k})\in a_{k+1}M.$$  
Hence, we have an ascending chain 
$$a_N\leq_Ma_{N+1}\leq_M\cdots$$
in $M.$  Since $M$ is $\arn$, there exists $N'\geq N$ such that $a_p\,\ar_M\,a_{N'}$ for all $p\geq N'.$  Therefore, for each $p\geq n$ there exists $m_p'\in M$ such that $a_p=a_{N'}m_p'.$  Then 
$$u_p=(i, a_p, j_p)=(i, a_{N'}, j_{N'})(j_{N'}, m_p', j_p)=u_{N'}(j_{N'}, m_p', j_p)\in u_{N'}N.$$
We conclude that $u_p\,\ar_N\,u_{N'}$ for all $p\geq N'.$  This completes the proof.
\end{proof}

\begin{cor}
\label{cor:embedding}
Every semigroup $S$ that is $\arn$ embeds into a simple semigroup that is $\arn$.
\end{cor}

\begin{proof}
Let $\theta : S^1\to S^1$ be the endomorphism given by $s\theta=1$ for all $s\in S^1$, and let $M=BR(S^1, \theta).$  Then $M$ is simple by \cite[Proposition 5.6.6(1)]{Howie}.  The monoid $S^1$ is $\arn$ since $S$ is, and hence $M$ is $\arn$ by Theorem \ref{thm:BR}.  We have already observed that $S^1$ is isomorphic to $\{0\}\times S^1\times\{0\}\subseteq M,$ and clearly $S$ embeds into $S^1,$ so we conclude that $S$ embeds into $M.$
\end{proof}

\section{\large{Semilattices of Semigroups}\nopunct}
\label{sec:sos}

Let $Y$ be a semilattice and let $(S_\a)_{\a\in Y}$ be a family of disjoint semigroups, indexed by $Y.$
If $S=\bigcup_{\a\in Y}S_\a$ is a semigroup such that $S_{\a}S_{\b}\subseteq S_{\a\b}$ for all $\a, \b\in Y,$
then $S$ is called a {\em semilattice of semigroups}, and we denote it by $S=\mathcal{S}(Y, S_\a).$  If, additionally, each $S_{\a}$ is a monoid, we call $S$ a {\em semilattice of monoids}.\par 
Now let $S=\bigcup_{\a\in Y}S_\a,$ and suppose that for each $\a, \b\in Y$ with $\a\geq\b$ there exists a homomorphism $\phi_{\a,\b} : S_{\a}\to S_{\b}.$  Furthermore, assume that:
\begin{itemize}
\item for each $\a\in Y,$ the homomorphism $\phi_{\a,\a}$ is the identity map on $S_{\a}$;
\item for each $\a, \b, \g\in Y$ with $\a\geq\b\geq\g$, we have $\phi_{\a, \b}\,\phi_{\b, \g}=\phi_{\a, \g}.$
\end{itemize}
For $a\in S_{\a}$ and $b\in S_{\b},$ we define
$$ab=(a\phi_{\a, \a\b})(b\phi_{\b, \a\b}).$$
With this multiplication, $S$ is a semilattice of semigroups.
In this case we call $S$ a {\em strong semilattice of semigroups} and denote it by $S=\mathcal{S}(Y, S_{\a}, \phi_{\a, \b}).$

We begin by investigating the behaviour of the property of being $\ar$-noetherian in the general setting of semilattices of semigroups.  Note that a semilattice is $\arn$ if and only if it satisfies the ascending chain condition on elements under its partial order.

\begin{defn}
Let $S=\S(Y,S_\a).$  We say that a chain $$\a_1\leq\a_2\leq\cdots$$ in $Y$ is $\ar${\em-witnessed} (with respect to $S$) if there exist $a_i\in S_{\a_i}$ ($i\in\N$) such that 
$$a_1\leq_Sa_2\leq_S\cdots.$$
\end{defn}

We note that if $S=\S(Y,S_\a)$ and $Y$ is $\arn,$ then certainly every $\ar$-witnessed chain in $Y$ with respect to $S$ eventually stabilises.  It turns out, however, that the converse does not hold in general.

\begin{lem}
\label{lem:sos}
Let $S=\S(Y,S_\a).$  If $S$ is $\arn,$ then every $\ar$-witnessed chain in $Y$ eventually stabilises and each $S_\a$ is $\arn$.
\end{lem}

\begin{proof}
Consider an $\ar$-witnessed chain $$\a_1\leq\a_2\leq\cdots$$ in $Y.$  Then there exist $a_i\in S_{\a_i}$ ($i\in\N$) such that
$$a_1\leq_Sa_2\leq_S\cdots.$$  
Since $S$ is $\arn$, there exists $n\in\N$ such that $a_n\,\ar_S\,a_N$ for all $n\geq N.$  This implies that $\a_n=\a_N$ for all $n\geq N,$ as required.

Now let $\a\in Y,$ and let $I=\bigcup_{\b\leq\a}S_{\b}.$  Then $I$ is an ideal of $S,$ so it is $\arn$ by Proposition \ref{prop:ideal}.  Since $I\!\setminus\!S_{\a}$ is an ideal of $I,$ it follows from Corollary \ref{cor:comp_left_ideal} that $S_{\a}$ is $\arn$.
\end{proof}

The following example shows that $\S(Y,S_\a)$ may not be $\arn$ even if $Y$ is finite and each $S_\a$ is $\arn.$  In particular, the converse of Lemma \ref{lem:sos} does not hold.

\begin{ex}
Let $S$ be the disjoint union of (a copy of) the free monogenic semigroup $\{x\}^+$ and a set $N=\{a_i : i\in\Z\}\cup\{0\}.$
Define a multiplication on $S,$ extending that on $\{x\}^+$, by
$$x^ia_j=a_jx^i=a_{j-i}\;\text{ and }\;x^i0=0x^i=uv=0\;(i\in\N, j\in Z, u, v\in N).$$
This multiplication turns $N$ into a null semigroup.  It is straightforward to show that, under this multiplication, $S$ is a semilattice of semigroups with structure semilattice $Y=\{1>0\}$ and corresponding components $\{x\}^+$ and $N.$  Certainly $S$ and $N$ are $\arn$.  On the other hand, it is easy to see that we have an infinite strictly ascending chain $$a_0<_Sa_1<_S\cdots$$ in $S,$ so that $S$ is not $\arn$.
\end{ex}

The following result provides a condition under which the converse of Lemma \ref{lem:sos} {\em does} hold. 
For this, recall that a semigroup is {\em weakly right noetherian} if it satisfies the ascending chain condition on right ideals.

\begin{prop}
Let $S=\S(Y, S_\a),$ and suppose that for each $\a\in Y$ the semigroup $S_{\a}$ contains no infinite antichain of $\ar_{S_\a}$-classes.  Then the following are equivalent:
\begin{enumerate}
\item $S$ is $\arn$;
\item every $\ar$-witnessed chain in $Y$ eventually stabilises and each $S_\a$ is $\arn$;
\item every $\ar$-witnessed chain in $Y$ eventually stabilises and each $S_\a$ is weakly right noetherian.
\end{enumerate}
\end{prop}

\begin{proof}
(1)$\Rightarrow$(2) follows from Lemma \ref{lem:sos}.

(2)$\Rightarrow$(1).  Suppose that every $\ar$-witnessed chain in $Y$ eventually stabilises but that $S$ is not $\arn$.  We need to prove that some $S_\a$ is not $\arn$.

Since $S$ is not $\arn$, there exists an infinite strictly ascending chain 
$$a_1<_Sa_2<_S\cdots$$
in $S.$  Let $a_i\in S_{\a_i},$ and for $i<j$ let $b_{i,j}\in S_{\b_{i, j}}$ be such that $a_i=a_jb_{i, j}$; then $\a_i=\a_j\b_{i,j}.$  Now, we have an $\ar$-witnessed chain 
$$\a_1\leq_Y\a_2\leq_Y\cdots$$ 
in $Y.$  By assumption, there exists $N\in\N$ such that $\a_n=\a_N$ for all $n\geq N.$  Then, letting $\a=\a_N$, we have $\a\b_{i,j}=\a$ for all $i, j\in\N$ with $N\leq i<j.$
 
Consider the set $\{a_n : n\geq N\}$ of elements of $S_\a$.  By assumption, this set cannot form an infinite antichain under the $\ar$-preorder on $S_\a$.  Also, we cannot have $a_j\leq_{S_\a}a_i$ for any $N\leq i<j,$ since this would contradict the fact that $a_i<_Sa_j$.  It follows that there exist $i_1, j_1\geq N$ with $i_1<j_1$ such that $a_{i_1}<_{S_\a}a_{j_1}$.  Now, either $i_1=N$ or
$$a_N=a_{i_1}b_{N,i_1}\in a_{j_1}(S_\a b_{N,i_1})\subseteq a_{j_1}S_{\a},$$
so $a_N<_{S_{\a}}a_{j_1}.$
Now consider the infinite set $\{a_n : n\geq j_1\}.$  By the same argument as above, there exists $j_2>j_1$ such that $a_{j_1}<_{S_\a}a_{j_2}$.
Continuing in this way, we obtain an infinite strictly ascending chain 
$$a_N<_{S_\a}a_{j_1}<_{S_\a}a_{j_2}<_{S_\a}\cdots$$
in $S_\a$, as required.

(2)$\Leftrightarrow$(3) follows from the fact that a semigroup is weakly right noetherian if and only if it is $\arn$ and contains no infinite antichain of $\ar$-classes \cite[Theorem 3.2]{Miller:2021}.
\end{proof}

The converse of Lemma \ref{lem:sos} also holds in the case that each $S_\a$ is $\ar$-preserving in $S.$

\begin{thm}\label{thm:sos,R-preserving}
Let $S=\S(Y, S_{\a})$ where each $S_\a$ is $\ar$-preserving in $S.$  Then $S$ is $\arn$ if and only if every $\ar$-witnessed chain in $Y$ eventually stabilises and each $S_\a$ is $\arn$.
\end{thm}

\begin{proof}
The forward implication follows immediately from Lemma \ref{lem:sos}
For the converse, consider an ascending chain $$a_1\leq_Sa_2\leq_S\cdots$$ in $S.$  Let $a_i\in S_{\a_i}.$  Then 
$$\a_1\leq_Y\a_2\leq_Y\cdots$$
is an $\ar$-witnessed chain in $Y.$  By assumption, there exists $N\in\N$ such that $\a_n=\a_N$ for all $n\geq N.$  Let $\a=\a_N.$  Then $a_n\in S_\a$ for all $n\geq N.$  Since $S_\a$ is $\ar$-preserving in $S,$ we have
$$a_N\leq_{S_\a}a_{N+1}\leq_{S_\a}\cdots.$$
Since $S_{\a}$ is $\arn$, there exists $N^{\prime}\geq N$ such that $a_n\,\ar_{S_\a}\,a_{N^{\prime}}$ for all $n\geq N^{\prime}$.  Then $a_n\,\ar_S\,a_{N^{\prime}}$ for all $n\geq N^{\prime}$.  This completes the proof.
\end{proof}

In what follows we consider certain situations where we have $S=\S(Y, S_{\a})$ with all the $S_\a$ being $\ar$-preserving in $S.$  
The first such situation is where every $S_{\a}$ has the property that each element has a local right identity (in $S_{\a}$).  Recall that this holds if each $S_{\a}$ is a monoid or regular semigroup.

\begin{prop}\label{prop:sos,lri}
Let $S=\S(Y,S_{\a})$ where, for each $\a\in Y,$ every element of $S_{\a}$ has a local right identity in $S_{\a}$.  Then $S$ is $\arn$ if and only if every $\ar$-witnessed chain in $Y$ eventually stabilises and each $S_\a$ is $\arn$.
\end{prop}

\begin{proof}
We show that each $S_\a$ is $\ar$-preserving, and the result then follows from Theorem \ref{thm:sos,R-preserving}.  So, let $\a\in Y,$ and let $a, b\in S_\a$ be such that $a\leq_Sb.$  Then $a=bs$ for some $s\in S^1.$  If $s=1$ then $a=b.$  Suppose that $s\in S.$  Then $s\in S_\b$ for some $\b\in Y,$ and we have $\a\b=\a.$  Let $c\in S_{\a}$ be a local right identity of $b,$ so that $b=bc.$  Then we have 
$$a=bs=(bc)s=b(cs)\in bS_\a,$$
so $a\leq_{S_\a}b,$ as required.
\end{proof}



A semigroup is called {\em completely regular} if it is a union of groups.  A semigroup is completely regular if and only if it is a semilattice of completely simple semigroups \cite[Theorem 4.1.3]{Howie}.  Completely simple semigroups are certainly $\arn$ by Corollary \ref{cor:min_right_ideals}.  Thus, by Proposition \ref{prop:sos,lri}, we have:

\begin{cor}\label{cor:cr}
Let $S$ be a completely regular semigroup, and let $S=\S(Y,S_\a)$ be its decomposition into a semilattice of completely simple semigroups.  Then $S$ is $\arn$ if and only if every $\ar$-witnessed chain in $Y$ eventually stabilises.
\end{cor}

The {\em free semilattice} on a non-empty set $X,$ which we denote by $F_X$, is defined as the set of all finite non-empty subsets of $X$ under the operation of union.  Clearly, for any $U, V\in F_X,$ we have $U\leq_{F_X}V$ if and only if $V\subseteq U.$  It follows that $F_X$ is $\arn$.  Both the free band on $X$ and free completely regular semigroup on $X$ are semilattices of semigroups where the structure semilattice is $F_X$; see \cite[p.\ 120]{Howie} and \cite[Corollary 4.3]{Clifford:1979}, respectively.  Therefore, by Corollary \ref{cor:cr}, we have:

\begin{cor}
Let $X$ be a non-empty set.  Then the following semigroups are $\arn$: the free semilattice on $X,$ the free band on $X,$ and the free completely regular semigroup on $X.$
\end{cor}

\begin{cor}\label{cor:som}
Let $S=\S(Y,S_\a)$ be a semilattice of monoids such that $1_\a1_\b=1_{\a\b}$ for all $\a,\b\in Y$ (where $1_\a$ denotes the identity of $S_\a$).  Then $S$ is $\arn$ if and only if $Y$ and all $S_\a$ are $\arn.$
\end{cor}

\begin{proof}
Given Proposition \ref{prop:sos,lri}, it suffices to show that every ascending chain of elements of $Y$ is $\ar$-witnessed with respect to $S.$  This follows from the fact that if $\a\leq\b$ then $1_\a=1_{\b\a}=1_\b1_\a,$ and hence $1_\a\leq 1_\b$.
\end{proof}

A {\em Clifford semigroup} is an inverse completely regular semigroup, or, equivalently, a (strong) semilattice of groups.  By Corollary \ref{cor:som} we have:

\begin{cor}
Let $S$ be a Clifford semigroup, and let $S=\S(Y,G_\a)$ be its decompostion into a semilattice of groups.  Then $S$ is $\arn$ if and only if $Y$ is $\arn$.
\end{cor}

We now turn our attention to strong semilattices of semigroups.  

\begin{prop}\label{prop:ssos}
Let $S=\S(Y,S_\a,\phi_{\a,\b}).$  Then $S$ is $\arn$ if and only if every $\ar$-witnessed chain in $Y$ eventually stabilises and each $S_\a$ is $\arn.$
\end{prop}

\begin{proof}
We show that each $S_\a$ is $\ar$-preserving, and the result then follows from Theorem \ref{thm:sos,R-preserving}.  So, let $\a\in Y,$ and let $a, b\in S_\a$ be such that $a\leq_Sb.$  Then $a=bs$ for some $s\in S^1.$  If $s=1$ then $a=b.$  Suppose that $s\in S.$  Then $s\in S_\b$ for some $\b\in Y,$ and we have $\a\b=\a.$  Thus, we have $a=b(s\phi_{\b,\a})\in b S_{\a}$, and hence $a\leq_{S_\a}b,$ as required.
\end{proof}

\begin{cor}\label{cor:ssos:nolri}
Let $S=\S(Y,S_\a,\phi_{\a,\b})$ where, for each $\a\in Y,$ $S_\a$ has no element with a local right identity.  Then $S$ is $\arn$ if and only if each $S_\a$ is $\arn$.
\end{cor}

\begin{proof}
Given Proposition \ref{prop:ssos}, it suffices to prove that if each $S_\a$ is $\arn$ then every $\ar$-witnessed chain in $Y$ eventually stabilises.  So, consider an $\ar$-witnessed chain $$\a_1\leq_Y\a_2\leq_Y\cdots.$$  Then there exist $a_i\in S_{\a_i}$ ($i\in\N$) such that
$$a_1\leq_Sa_2\leq_S\cdots.$$
By the definition of the product in $S,$ for each $i\in\N$ we have either $a_i=a_{i+1}$ or $a_i=(a_{i+1}\phi_{\a_{i+1},\a_i})s_i$ for some $s_i\in S_{\a_i}$.  If $a_i=a_{i+1}$ then $a_i\phi_{\a_i,\a_1}=a_{i+1}\phi_{\a_{i+1},\a_1}$, and if $a_i=(a_{i+1}\phi_{\a_{i+1},\a_i})s_i$ then
$$a_i\phi_{\a_i,\a_1}=(a_{i+1}\phi_{\a_{i+1},\a_1})(s_i\phi_{\a_i,\a_1}).$$ 
Thus, letting $T=S_{\a_1},$ we have an ascending chain $$a_1=a_1\phi_{\a_1,\a_1}\leq_Ta_2\phi_{\a_2,\a_1}\leq_T\cdots$$
in $T.$  Since $T$ is $\ar$-trivial (as it has no element with a local right identity) and $\arn$, there exists $N\in\N$ such that $a_n\phi_{\a_n,\a_1}=a_N\phi_{\a_N,\a_1}$ for all $n\geq N.$  We claim that $\a_n=\a_N$ for all $n\geq N.$  Indeed, suppose not. Then there exists $n\geq N$ such that $\a_n\neq\a_{n+1}$, and hence $a_n\neq a_{n+1}.$  Then we have $$a_n\phi_{\a_n,\a_1}=(a_{n+1}\phi_{\a_{n+1},\a_1})(s_n\phi_{\a_n,\a_1}).$$  But then $$a_N\phi_{\a_N,\a_1}=(a_N\phi_{\a_N,\a_1})(s_n\phi_{\a_n,\a_1})\in(a_N\phi_{\a_N,\a_1})T,$$ contradicting the fact that $T$ has no element with a local right identity.  This completes the proof.
\end{proof}

We conclude this section with an example of a strong semilattice of semigroups that is $\arn$ but whose structure semilattice is not $\arn$. 

\begin{ex}
Let $Y$ be the semilattice $(\N,\text{min}).$  Then $Y$ is not $\arn.$  For each $i\in\N,$ let $S_i$ be the semilattice $\{i\}\times\N$ with multiplication 
$$(i,j)(i,k)=\big(i,\max(j,k)\big)\:\text{for all}\:j,k\in\N.$$
Then each $S_i$ is isomorphic to $(\N,\text{max}),$ which is certainly $\ar$-noetherian.
For each $i,j\in\N$ with $i\geq j,$ define a map 
$$\phi_{i,j} : S_i\to S_j, (i,n)\mapsto(j,n+i-j).$$
For any $m,n\in\N,$ we have
\begin{align*}
\big((i,m)(i,n)\big)\phi_{i,j}&=\big(i,\max(m,n)\big)\phi_{i,j}=\big(j,\max(m,n)+i-j\big)\\
&=\big(j,\max(m+i-j,n+i-j)\big)=(j,m+i-j)(j,n+i-j)\\
&=\big((j,m)\phi_{i,j}\big)\big((j,n)\phi_{i,j}\big),
\end{align*}
so $\phi_{i,j}$ is a homomorphism.  Let $S$ be the strong semilattice of semilattices $\S(Y,S_i,\phi_{i,j}).$  Then $S=\N\times\N,$ and for any $(i,m),(j,n)\in S$ we have
\begin{align*}
(i,m)(j,n)&=\big((i,m)\phi_{i,\min(i,j)}\big)\big((j,n)\phi_{j,\min(i,j)}\big)\\
&=\big(\min(i,j),m+i-\min(i,j)\big)\big(\min(i,j),n+j-\min(i,j)\big)\\
&=\big(\min(i,j),\max(m+i-\min(i,j),n+j-\min(i,j))\big).
\end{align*}
It is easy to see that this product is commutative.  Therefore, since each element of $S$ is an idempotent, we have that $S$ is a semilattice.  We claim that $S$ is $\arn$.  So, consider an ascending chain $$(i_1,n_1)\leq_S(i_2,n_2)\leq_S\cdots$$ in $S.$  Since $S$ is a semilattice, for each $j\in\N$ we have $(i_j,n_j)=(i_{j+1},n_{j+1})(i_j,n_j).$  Thus, we have $i_j=\min(i_{j+1},i_j)$ and $n_j=\max(i_{j+1}+n_{j+1}-i_j, n_j).$
It follows that
$$i_1\leq i_2\leq\cdots\qquad\text{and}\qquad n_1\geq n_2\geq\cdots,$$
and if $i_j<i_{j+1}$ then $n_j>n_{j+1}$.  We conclude that there exists $N\in\N$ such that $i_m=i_N$ and $n_m=n_N$ for all $m\geq N,$ and hence $(i_m,n_m)=(i_N,n_N)$ for all $m\geq N,$ as required.
\end{ex}

\section*{Funding}
This work was supported by the Engineering and Physical Sciences Research Council [EP/V002953/1].

\section*{Acknowledgements}
The author thanks the referee for a careful reading of the paper and for a number of helpful suggestions, especially those that led to Theorem \ref{thm:sos,R-preserving} and Proposition \ref{prop:sos,lri}.


\end{document}